\documentclass{amsart}
\usepackage{graphicx}
\usepackage{verbatim}
\usepackage{cite}
\newtheorem{thm}{Theorem}[section]
\newtheorem{cor}[thm]{Corollary}
\newtheorem{lem}[thm]{Lemma}
\newtheorem{prop}[thm]{Proposition}
\theoremstyle{definition}
\newtheorem{defn}[thm]{Definition}
\theoremstyle{remark}
\newtheorem{rem}[thm]{Remark}
\numberwithin{equation}{section}

\newcommand{\set}[1]{\left\{#1\right\}}
\newcommand{\Real}{\mathbb R}

\newcommand{\func}[1]{\ensuremath{\mathrm{#1} \:} }
\newcommand{\Hess}[0]{\func{Hess}}
\newcommand{\Div}[0]{\func{div}}
\newcommand{\pRect}[4]{S_{{#1}, {#2}}^{{#3}, {#4}}}

\newcommand{\dist}[0]{\mathrm{dist}}
\newcommand{\re}[0]{\func{Re}}
\newcommand{\im}[0]{\func{Im}}
\newcommand{\osc}[0]{\mathop{\mathrm{osc}}}
\title{Helicoid-Like Minimal Disks and Uniqueness}
\author{Jacob Bernstein and Christine Breiner}
\address{Dept. of Math, Stanford University, Stanford, CA  94305, USA}
\email{jbern@math.stanford.edu}

\address{Dept. of Math,
MIT, Cambridge, MA 02139, USA}
\email{breiner@math.mit.edu}

\thanks{The first author was partially supported by NSF grant DMS 0606629.}

\begin{document}
\begin{abstract}
 We show that for an embedded minimal disk in $\Real^3$, near points of large curvature the surface is bi-Lipschitz with a piece of a helicoid. Additionally, a simplified proof of the uniqueness of the helicoid is provided.
 \end{abstract}
\maketitle
\section{Introduction}
This paper gives a condition for an embedded minimal disk to look like a piece of a helicoid.  Namely, if such a disk has
boundary in the boundary of a ball and has large curvature, then, in
a smaller ball, it is bi-Lipschitz to a piece of a helicoid.
Moreover, the Lipschitz constant can be chosen as close to 1 as
desired (compare with
 Proposition 2 of \cite{M1}).
\begin{thm}
\label{FifthMainThm} Given $\epsilon,R>0$ there exists $R'\geq R$
so: Suppose $0\in \Sigma' \in \mathcal{E}(1,0, R's)$, and $(0,s)$ is
a blow-up pair (see section
 \ref{sheetssection}).  Then there
exists $\Omega$, a subset of a helicoid, so that $\Sigma$, the
component of $\Sigma'\cap B_{Rs}$ containing 0, is bi-Lipschitz with
$\Omega$, and the Lipschitz constant is in
$(1-\epsilon,1+\epsilon)$.
\end{thm}
Here $\mathcal{E}(1,g,R)$ denotes the space of embedded minimal
surfaces $\Sigma \subset \Real^3$ of genus $g$ and with smooth, connected
boundary, $\partial \Sigma \subset \partial B_R(0)$.  We say a
surface has genus $g$ if it is diffeomorphic to a punctured,
compact, oriented genus $g$ surface, though in this paper we restrict attention to disks, i.e. $g=0$.

Colding and Minicozzi, in their work on the shapes of
embedded minimal disks \cite{CM1,CM2,CM3,CM4} (see also
\cite{CMPNAS} for a non-technical overview), show that a $\Sigma$ as
in the above theorem looks, on a scale relative to $R$, roughly like
the helicoid -- an essentially qualitative description. Theorem
\ref{FifthMainThm} sharpens this description, though on a much
smaller scale, giving a quantitative description of such a disk near
a point of large curvature. Examples constructed by Colding and
Minicozzi \cite{CMPVNP}, Khan \cite{K}, Kleene \cite{Kl}, and Hoffman and White \cite{HW3} demonstrate that
this sharper description cannot hold on the outer scale $R$  -- we refer the interested reader to \cite{BB2}.

To prove Theorem \ref{FifthMainThm} we argue by contradiction, coupling the lamination theory of Colding and Minicozzi of \cite{CM4} with the so-called uniqueness of the
helicoid:
\begin{thm}
\label{FourthMainThm} Any $\Sigma \in \mathcal{E}(1,0)$ is either
a plane or a helicoid.
\end{thm}
Here $\mathcal{E}(1,g,\infty)= \mathcal{E}(1,g)$ is the space of
complete, embedded minimal surfaces in $\Real^3$ with genus $g$ and
one end. Notice there is no a priori assumption that elements
are properly embedded.  This is because Colding
 and Minicozzi prove in Corollary 0.13 of \cite{CY} that a complete, embedded minimal surface of finite topology is automatically properly embedded, a fact we use throughout.

Theorem \ref{FourthMainThm} was first proved by Meeks and Rosenberg in
\cite{MR}.  Their argument uses, in an essential manner, the theory
of Colding and Minicozzi, in particular the lamination theory and
one-sided curvature estimate of \cite{CM4}. In this paper, we will
provide a new and more geometric proof of this fundamental theorem.  Our
argument makes direct use of the results of Colding and Minicozzi on
the geometric structure of embedded minimal disks from
\cite{CM1,CM2,CM3,CM4}. Importantly, the geometric decomposition
described in Theorem \ref{FirstMainThm} below, and hence the entire proof, can be extended to arbitrary
$\Sigma \in \mathcal{E}(1,g)$ for $g>0$.  Indeed, in \cite{BB3} we prove
the following generalization of Theorem \ref{FourthMainThm}:
\begin{thm}
Let $\Sigma\in \mathcal{E}(1,g)$. Then $\Sigma$ is
\textit{conformally} a once punctured, compact Riemann surface.
Moreover, if $\Sigma$ is non-flat, it is asymptotic to a helicoid.
\end{thm}

In their paper, Meeks and Rosenberg first use the lamination
theory to deduce that (after a rotation) a homothetic blow-down of a
non-flat $\Sigma\in \mathcal{E}(1,0)$ is, away from some
Lipschitz curve, a foliation of flat parallel planes transverse to
the $x_3$-axis. This gives, in a very weak sense, that the surface is
asymptotic to a helicoid, which they use to conclude that the
Gauss map of $\Sigma$ omits the north and south poles. The
asymptotic
 structure,  a result of Collin, Kusner, Meeks and Rosenberg regarding the parabolicity of minimal surfaces \cite{CKMR} and some very delicate complex analytic arguments, are combined to show that $\Sigma$ is conformally equivalent to $\mathbb{C}$.
Finally, by looking at the level sets of the $\log$ of the stereographic
projection of the Gauss map and using
 a Picard
type argument, they show that this holomorphic map does not have an
essential singularity at $\infty$ and is, in fact, linear.  The
Weierstrass representation then implies $\Sigma$ is the helicoid.

The geometric approach of our paper allows for a more direct
argument.  Recall, the helicoid contains a central ``axis'' of large
curvature away from which it consists of two multivalued graphs
spiraling together, one strictly upward, the other downward.  Note that the known embedded genus-one helicoids -- constructed by Weber, Hoffman and Wolf \cite{WHW}  and
by Hoffman and White \cite{HW} -- behave similarly.  We first show that this is the structure
of any non-flat $\Sigma \in \mathcal{E}(1,0)$:
\begin{thm}
\label{FirstMainThm} There exist $\epsilon_0>0$ and disjoint subsets of $\Sigma$,
$\mathcal{R}_A$ and $\mathcal{R}_S$, with $\Sigma=\mathcal{R}_A\cup
\mathcal{R}_S$ such that:
\begin{enumerate}
\item after possibly rotating $\Real^3$,
$\mathcal{R}_S$ can be written as the union of two (oppositely
oriented) strictly spiraling multivalued graphs $\Sigma^1$ and $\Sigma^2$; 
\item in $\mathcal{R}_A$, $|\nabla_\Sigma x_3|\geq \epsilon_0$.
\end{enumerate}
\end{thm}
\begin{rem}
 We say $\Sigma^i$ ($i=1,2$) is a multivalued graph if it can be decomposed into $N$-valued
 $\epsilon$-sheets (see Definition \ref{epsilonsht}) with varying
 center. That is,
$\Sigma^i=\cup_{j=-\infty}^\infty \Sigma_j^i$ where each
$\Sigma_j^i=y^i_j+\Gamma_{u^i_j}$ is an $N$-valued $\epsilon$-sheet.
Strict spiraling then means that $(u_j^i)_\theta\neq 0$ for
all $j$.  A priori, the axes of the multivalued graphs vary, a fact
that introduces additional book-keeping.  For the sake of clarity,
we assume that each $\Sigma^i$ is an $\infty$-valued
$\epsilon$-sheet -- i.e $\Sigma^i$ is the graph, $\Gamma_{u^i}$, of
a single $u^i$ with $u^i_\theta\neq 0$.
\end{rem}

In order to establish this decomposition, we first use Colding and
Minicozzi's results on the structure of embedded minimal disks to
obtain the existence of two infinite-sheeted multivalued graphs,
$\Sigma^1,\Sigma^2,$ that spiral together.  Then, using a result of
\cite{EXC}, we show that such graphs can be approximated
asymptotically in a manner that allows one to show that far enough
out along each sheet of the multivalued graphs, the sheet strictly
spirals -- giving the region $\mathcal{R}_S$. 
An application of the proof of Rado's theorem
\cite{RadoThm} then implies $|\nabla_\Sigma x_3|\neq 0$ on
$\mathcal{R}_A$, the subset of $\Sigma$ away from the two
multivalued graphs.  Finally, a Harnack inequality gives the
uniform lower bound.  In showing this, one obtains:
\begin{prop}
\label{SecondMainThm} On $\Sigma$, after a rotation of $\Real^3$,
$\nabla_\Sigma x_3 \neq 0$ and, for all $c\in \Real$, $\Sigma\cap
\set{x_3=c}$ consists of exactly one properly embedded smooth curve.
\end{prop}
Thus, $z=x_3+ix_3^*$ is a holomorphic coordinate on
$\Sigma$. Using the stereographic projection of the Gauss
map, $g$, we show that $z$ maps onto
$\mathbb{C}$ and so $\Sigma$ is conformally the plane.
 This follows from control on the behavior of $g$ in $\mathcal{R}_S$ due to the strict
 spiraling.  Indeed, away from a small neighborhood of
$\mathcal{R}_A$, $\Sigma$ is conformally the union of two closed
half-spaces with $\log g=h$ providing the identification.  It then
 follows that $h$ is also a
 conformal diffeomorphism which gives Theorem \ref{FourthMainThm}.

\section{Global Geometric Structure of $\Sigma$}
To study elements of $\mathcal{E}(1,0)$ we rely heavily on Colding
and Minicozzi's structural results for embedded minimal disks.
Much of this can be found in the series of papers
\cite{CM1,CM2,CM3,CM4}, with more technical analysis in
\cite{MMGPD}.  For a detailed overview of the theory, the
interested reader should consult the survey \cite{CMEMD}.
We have gathered the major results we use in Appendix \ref{structureapp}.  Additionally, for the convenience
of the reader we refer, when possible, to
Appendix \ref{structureapp} rather then directly to the papers of
Colding and Minicozzi.

\subsection{Preliminaries}
Throughout, let $\Sigma \in \mathcal{E}(1,0)$ be non-flat. Recall
$\Sigma \in \mathcal{E}(1,g,R)$ is an embedded minimal surface with
genus $g$ and so that $\partial \Sigma \subset
\partial B_R(0)$ is connected and $\Sigma \in \mathcal{E}(1,g)$ is a complete, embedded minimal surface with genus $g$ and one end. Here $B_r(y)$ is the Euclidean ball
of radius $r$ centered at $y$; for a point $p\in\Sigma$ we denote an
intrinsic ball in $\Sigma$ of radius $R$ centered at $p$ by
$\mathcal{B}_R(p)$. We let $|A|^2$ represent the norm squared of the
second fundamental form on $\Sigma$.
 Let
\begin{equation}
\mathbf{C}_\delta (y)=\set{x: (x_3-y_3)^2 \leq \delta^2
\big((x_1-y_1)^2 +(x_2-y_2)^2\big)}\subset \Real^3
\end{equation}
be the complement of a cone and set
$\mathbf{C}_\delta=\mathbf{C}_\delta (0)$. We denote a polar
rectangle by:
\begin{equation}
\pRect{r_1}{r_2}{\theta_1}{\theta_2}=\set{(\rho,\theta) \mid
r_1\leq \rho\leq r_2, \theta_1\leq \theta \leq \theta_2}.
\end{equation}
For a real-valued function, $u$, defined on a polar domain
$\Omega\subset \Real^+\times \Real$, define the map $\Phi_u
:\Omega\to \Real^3$ by $\Phi_u (\rho,\theta)=(\rho \cos
\theta,\rho \sin \theta, u(\rho,\theta))$. In particular, if $u$
is defined on $\pRect{r_1}{r_2}{\theta_1}{\theta_2}$, then
$\Phi_u(\pRect{r_1}{r_2}{\theta_1}{\theta_2})$ is a multivalued
graph over the annulus $D_{r_2}\backslash D_{r_1}$.  We define the
\emph{separation} of the graph $u$ by $w(\rho,\theta) = u(\rho,\theta +
2\pi) - u(\rho,\theta)$.  Thus, $\Gamma_u := \Phi_u (\Omega)$ is
the graph of $u$, and $\Gamma_u$ is embedded if and only if $w\neq
0$.  Finally, we say a graph $\Gamma_u$ \emph{strictly spirals} if $u_\theta \neq 0$.

Recall that $u$ satisfies the minimal surface equation if:
\begin{equation} \label{MinSurfEq}
\Div \left(\frac{\nabla u}{\sqrt{1+|\nabla u|^2}}\right)=0.
\end{equation}
The graphs of interest to us will also satisfy the following flatness condition:
\begin{equation}
\label{MainInEq21} |\nabla u|+\rho |\Hess_u|+4 \rho \frac{|\nabla
w|}{|w|}+\rho^2
 \frac{|\Hess_w|}{|w|}\leq \epsilon <\frac{1}{2\pi}.
\end{equation}
Note that if $w$ is the separation of a $u$ satisfying
\eqref{MinSurfEq} and \eqref{MainInEq21}, then $u$ and $w$ satisfy
uniformly elliptic second order equations.  Thus, if $\Gamma_u$ is embedded then
$w$ has point-wise gradient bounds and satisfies a Harnack
inequality.

In Colding and Minicozzi's work, multivalued minimal graphs are a
 basic building block used to study the structure of minimal surfaces.
 We also make heavy use of them and so introduce some notation.
\begin{defn}\label{weakepsilonsht}  A multivalued minimal graph $\Sigma_0$ is a \emph{weak} \emph{$N$-valued} (\emph{$\epsilon$}-)\emph{sheet} (\emph{centered at $y \in \Sigma$ on the scale $s>0$}) if $\Sigma_0=\Gamma_{u}+y$ and $u$, defined on
$\pRect{s}{\infty}{-\pi N}{ \pi N}$, satisfies  \eqref{MinSurfEq},
has $|\nabla u|\leq \epsilon$, and
$\Sigma_0\subset \mathbf{C}_{\epsilon}(y)$.
\end{defn}
We will often need more control on the sheets as well as a normalization at $\infty$:
\begin{defn} \label{epsilonsht} A multivalued minimal graph $\Sigma_0$ is a (\emph{strong}) \emph{$N$-valued} (\emph{$\epsilon$}-)\emph{sheet} (\emph{centered at $y$ on the scale $s$}), if $\Sigma_0=\Gamma_{u}+y$ is a weak $N$-valued $\epsilon$-sheet centered at $y$ and on scale $s$ and, in addition, $u$ satisfies
\eqref{MainInEq21} and $\lim_{\rho\to \infty} \nabla u(\rho,0)=0$.
\end{defn}
Using Simons' inequality, Corollary 2.3 of \cite{MMGPD} shows that
on the
one-valued middle sheet of a 2-valued graph satisfying
\eqref{MainInEq21}, the hessian of $u$ has faster than linear
decay and hence $u$ has an asymptotic tangent plane. This Bers-like
result implies that the normalization at $\infty$ in the
definition of an $\epsilon$-sheet is well defined. As an additional consequence, for an $\epsilon$-sheet,
\begin{equation}
  \label{GradDecay}
  |\nabla u| \leq  C\epsilon \rho^{-5/12}.
\end{equation}
Notice that if $\Sigma_0=\Gamma_u$ is a weak $\epsilon$-sheet then
because $u$ has bounded gradient and satisfies \eqref{MinSurfEq},
it is the solution to a uniformly elliptic second order equation.
As such, one can apply standard elliptic estimates to gain further
(interior) regularity.  In particular, looking on a sub-graph (and
possibly slightly rotating to normalize behavior at $\infty$) one
has that $\Sigma_0$ contains an $\epsilon$-sheet:
\begin{prop}
\label{WeakToEpsilonShtPrp} Given $N\in \mathbb{Z}^+$ and $\epsilon>0$, sufficiently small (depending on $N$), there
exist $N_0$ and $C$ depending only on $\epsilon$ and $N$ so:  Suppose that $\Sigma=\Gamma_u$ is a weak $N_0$-valued $\epsilon/C$-sheet centered at $0$ and on scale $s/C$.  Then (possibly after a small rotation) $\Sigma$ contains an $N$-valued $\epsilon$-sheet $\Sigma_0=\Gamma_{u_0}$
on the scale $s$.
\end{prop}
\begin{proof}
Proposition II.2.12 of \cite{CM1} and standard elliptic estimates
give an $N_\epsilon\in \mathbb{Z}^+$ and $\delta_\epsilon>0$
depending only on $\epsilon$ so that if $u$ satisfies
\eqref{MinSurfEq} and $|\nabla u|\leq \epsilon/4$ on
$\pRect{e^{-N_\epsilon}}{\infty}{-\pi N_\epsilon}{\pi N_\epsilon}$
and $\Gamma_u\subset \mathbf{C}_{\delta_\epsilon}$, then on
$\pRect{1}{\infty}{0}{2\pi}$ we have the sum of all the terms of
\eqref{MainInEq21} bounded by $\epsilon/2$. Hence by the above (and a
rescaling) we see that for $N_0=N_\epsilon+N$ and
$C=\max\set{2e^{N_\epsilon}, \delta_\epsilon^{-1}}$,  $u$
satisfies \eqref{MainInEq21} on $\pRect{s/2}{\infty}{-\pi N}{\pi
N}$.  At this point we do not a priori know that $\lim_{\rho\to
\infty} \nabla u (\rho,0)=0$.  However, there is an asymptotic
tangent plane. Thus, we may need  a small rotation (the size of
which is controlled by $\epsilon$) to make this parallel to the
$x_1$-$x_2$ plane.  Notice this rotation affects the inner scale
and the bound on \eqref{MainInEq21}.  Nevertheless, for small
enough rotations we may replace $s/2$ by $s$ and obtain
$\Sigma_0$.
\end{proof}
\subsection{Initial Sheets} \label{sheetssection}
We now use Colding and Minicozzi's work to establish the existence
of a suitable number of $\epsilon$-sheets within $\Sigma$. 
First, we give a condition for the existence of $\epsilon$-sheets.
Roughly, all that is required is a point with large curvature
relative to nearby points. This is made precise by:
\begin{defn}\label{defbup} The pair $(y,s)$, $y\in \Sigma$, $s>0$, is a
 (\emph{$C$}) \emph{blow-up pair} if
\begin{equation}
\sup_{\Sigma\cap B_s(y)} |A|^2\leq 4|A|^2(y)=4C^2 s^{-2}.
\end{equation}
\end{defn}
The existence of a blow-up pair in an embedded minimal disk forces the
surface to spiral nearby and this extends outward (see Theorem \ref{CM204Thm}). As a
consequence, after a suitable rotation, we obtain a weak sheet near
the pair; by Proposition \ref{WeakToEpsilonShtPrp} this contains an
 $\epsilon$-sheet. Hence, near any blow-up pair there is an $\epsilon$-sheet:
\begin{thm}
\label{InitShtExst} Given $N\in \mathbb{Z}^+$ and $\epsilon>0$ sufficiently small, there
exist $C_1, C_2>0$ so:  Suppose that $(0,s)$ is a $C_1$ blow-up
pair of $\Sigma$.  Then there exists (after a rotation of
$\Real^3$) an $N$-valued $\epsilon$-sheet $\Sigma_0=\Gamma_{u_0}$
on the scale $s$.
  Moreover,
the separation over $\partial D_s$ of $\Sigma_0$ is bounded below
by $C_2 s$.
\end{thm}
\begin{proof}
Using $\epsilon$ and $N$, let $N_0$ and $C$ be given by Proposition \ref{WeakToEpsilonShtPrp}.  Then to find the desired $\epsilon$-sheet we must ensure the existence of a weak $N_0$-valued $\epsilon/C$-sheet near $0$.

To that end, use $\epsilon/{C}$ and $N_0$ with Theorem \ref{CM202Thm} to obtain
$C_1', C_2'>0$.
 That is, if $(0,t/2)$ is a
$C_1'$ blow-up pair in $\Sigma$, then (up to rotating $\Real^3$) there is a weak $N_0$-valued $\epsilon/C$-sheet centered at $0$ and on scale $t$.  Thus, (up to a further small rotation) one has an $N$-valued $\epsilon$-sheet, $\Sigma_0=\Gamma_{u_0}$, centered at $0$ and on scale $Ct$.
The proof of Proposition 4.15 of \cite{CM2} provides a constant $C_2>0$  (depending only on $C$) so that $w_0
(Ct,\theta)\geq C_2 C t$. Finally, if we
set $C_1=2 C_1' C$ then $(0,s)$ being a $C_1$ blow-up
pair implies that $(0,\frac{s}{2C})$ is a $C_1'$ blow-up
pair. This gives the result.
\end{proof}


Once there is one $\epsilon$-sheet, $\Sigma_1$, in $\Sigma$, a barrier
argument shows that between the sheets of $\Sigma_1$, $\Sigma$
consists of exactly one other $\epsilon$-sheet. Namely, by Theorem
\ref{CM4I010Thm}, the parts of $\Sigma$ that lie in between an
$\epsilon$-sheet make up a second multivalued graph. Furthermore,
the one-sided curvature estimate of \cite{CM4} (See Appendix
\ref{oscsec}) gives gradient estimates which, coupled with
Proposition \ref{WeakToEpsilonShtPrp}, reveal that this graph contains an $\epsilon$-sheet. Thus, near a blow-up point,
$\Sigma$ contains two $\epsilon$-sheets spiraling together.

We now make the last statement precise. Suppose $u$ is defined on
$\pRect{1/2}{\infty}{-\pi N-3\pi}{\pi N+3\pi}$ and $\Gamma_u$ is
embedded.  We define $E$ to be the region over $D_\infty
\backslash D_1$ between the top and bottom sheets of the
concentric sub-graph of $u$.  That is:
\begin{multline}\label{Edefn}
E=\{(\rho\cos \theta, \rho\sin \theta, t): \\ 1\leq \rho\leq
\infty, -2\pi \leq \theta <0, u(\rho,  \theta-\pi N)<t<u(\rho,
\theta+ (N+2)\pi \}.
\end{multline}
Using Theorem \ref{CM4I010Thm}, Theorem \ref{InitShtExst}, and
the one-sided curvature estimate:
\begin{thm} \label{TwoInitShtExstThm}
Given $\epsilon>0$ sufficiently small, there exist $C_1, C_2>0$
so: Suppose $(0,s)$ is a $C_1$ blow-up pair.  Then there exist two
$4$-valued $\epsilon$-sheets $\Sigma_i=\Gamma_{u_i}$ ($i=1,2$) on
the scale $s$ which spiral together (i.e. $u_1(s,0)<
u_2(s,0)<u_1(s,2\pi)$).  Moreover, the separation over $\partial
D_s$ of $\Sigma_i$ is bounded below by $C_2 s$.
\end{thm}
\begin{rem} We refer to $\Sigma_1, \Sigma_2$ as ($\epsilon$-)blow-up
 sheets \emph{associated
with} $(y,s)$.
\end{rem}
\begin{proof}
Fix $\epsilon_0>0$ as in Theorem \ref{CM4I010Thm}. For $\epsilon <
\epsilon_0$ and $N=4$, choose $N_0$ and $C$ as given by Proposition
\ref{WeakToEpsilonShtPrp}. With $\tilde{N}=10+N_0$ denote by $C_1',
C_2'$ the constants given by Theorem \ref{InitShtExst} (with
$\tilde{N}$ replacing the $N$ in the theorem). Thus, if $(0,r)$ is a
$C_1'$ blow-up pair then there exists an $\tilde{N}$-valued
$\epsilon$-sheet $\Sigma'_1=\Gamma_{u'_1}$ on scale $r$ inside of
$\Sigma$ with $w(r,\theta) \geq C_2' r$.  Applying Theorem
\ref{CM4I010Thm} to $u'_1$,
 we see that $\Sigma\cap E\backslash \Sigma_1'$ is given by the graph
 of
 a function $u_2'$ defined on $\pRect{2r}{\infty}{-\pi
 \tilde{N}+2\pi}{\pi \tilde{N}-2\pi}$. As long as we can control the gradient of $u'_2$ this will give us a weak sheet.

To that end, let $\alpha=\epsilon/C$; use the one-sided curvature
estimate (Corollary \ref{osccor}) to choose
$\epsilon/C>\delta_0>0$. By \eqref{GradDecay} and the fact that the initial $\tilde{N}$ sheets of $u'_1$ exist outside a cone, there exists $\tilde{C}>1$, depending on $\epsilon, \delta_0$, and the constant in \eqref{GradDecay},
such that $|\nabla u_1'| \leq \delta_0$ on
$\pRect{\tilde{C}r}{\infty}{-\pi\tilde{N} + \pi}{\pi\tilde{N} - \pi}$ and
$u_1'$ restricted to this domain is contained in
$\mathbf{C}_{\delta_0} \backslash B_{\tilde{C}r}$. Thus, the gradient of
$u_2'$ restricted to
$\pRect{2\tilde{C}r}{\infty}{-\pi\tilde{N}+3\pi}{\pi\tilde{N} - 3\pi}$ is
bounded by $\epsilon/C$. Moreover, since $\tilde{N}-1$ sheets of
$u_1'$ are inside of $\mathbf{C}_{\delta_0}$, the $\tilde{N}-3$
concentric sheets of $u_2'$ are also in $\mathbf{C}_{\delta_0}$.
Thus, the graph of $u_2'$ over $\pRect{2\tilde{C}r}{\infty}{-\pi N_0}{\pi
N_0}$ is a weak $N_0$-valued $\epsilon/C$ sheet centered at $0$
and on scale $2\tilde{C}r$. Proposition
\ref{WeakToEpsilonShtPrp} then gives that the graph of $u_2'$ over
$\pRect{2C\tilde{C}r}{\infty}{-4\pi}{4\pi}$ is an $\epsilon$-sheet.
Notice no rotation is needed here as $u_1'$ is already an
$\epsilon$-sheet.

Let $u_1$ and $u_2$ be given by
 restricting $u_1'$ and $u_2'$ to $\pRect{2C\tilde{C}
 r}{\infty}{-4\pi}{4\pi}$
 and define $\Sigma_i=\Gamma_{u_i}$.
 Set $C_1=2C\tilde{C} C_1'$, so if $(0,s)$ is a $C_1$ blow-up pair
  then $\Sigma_i$ will exist on scale $s$.
  Integrating \eqref{MainInEq21}, the lower bound $C_2'$
  gives a lower bound on initial separation of $\Sigma_1$.  We find
 $C_2$ by noting that if the initial separation of $\Sigma_2$ was too small
 there would be two sheets between one sheet of $\Sigma_1$.
\end{proof}
\subsection{Blow-up Pairs}
Since $\Sigma$ is not a plane, we can always find at least one
blow-up pair
$(y,s)$ -- an immediate consequence of Lemma \ref{CM251Lem}.  We then use this initial pair to find a sequence of
blow-up pairs forming an ``axis" of large curvature. The key
results we need are Lemma \ref{CM251Lem}, which says that as
long as curvature is large enough in a ball, measured relative to the scale of the ball, we can find a
blow-up pair in the ball, and Corollary \ref{CM3III35Cor},
which guarantees points of large curvature above and below blow-up
points. Colding and Minicozzi, in Lemma 2.5 of \cite{CY}, provide a good
overview of this process of decomposing $\Sigma$ into blow-up
sheets.  The main result is the following:
\begin{thm}
\label{BlowUpPrsThm} For $1/2>\gamma>0$  and $\epsilon>0$ both
sufficiently small, let $C_1$ be given by Theorem
\ref{TwoInitShtExstThm}.  Then there exists $C_{in}>4$ and
$\delta>0$ so: If $(0,s)$ is a $C_1$ blow-up pair then there exist
$(y_+,s_+)$ and $(y_-,s_-)$, $C_1$ blow-up pairs, with $y_\pm \in
\Sigma\cap B_{C_{in}s}\backslash \left(
B_{2s}\cup\mathbf{C}_\delta \right)$, $x_3(y_+)>0>x_3(y_-)$, and
$s_\pm \leq \gamma|y_{\pm}|$.
\end{thm}
Hence, given a blow-up pair, we can iteratively find a sequence of
blow-up pairs ordered by height and lying in a cone, with distance
between subsequent pairs bounded by a fixed multiple of the scale.
\section{Asymptotic Helicoids}
Lemma 14.1 of \cite{EXC} and the gradient decay \eqref{GradDecay}
shows
 that $\epsilon$-sheets can be
approximated by a combination of planar, helicoidal, and
catenoidal pieces.  Precisely, there is a ``Laurent expansion''
for the almost holomorphic function $u_x-iu_y$.  This
 result allows us to bound the
oscillation on broken circles
$C(\rho):=\pRect{\rho}{\rho}{-\pi}{\pi}$ of $u_\theta$, which
yields asymptotic lower bounds for $u_\theta$.
\begin{lem}\label{exclemma} Given $\Gamma_u$, a 3-valued
 $\epsilon$-sheet on scale 1, set $f=u_x-i
u_y$.  Then for $r_1 \geq 1$ and $\zeta =\rho e^{i\theta}$ with
$(\rho,\theta)\in \pRect{2r_1}{\infty}{-\pi}{\pi}$
\begin{equation}
f(\rho,\theta)=c \zeta^{-1}+g(\zeta)
\end{equation}
where $c =c(r_1,u)\in \mathbb{C}$ and $|g(\zeta)|\leq C_0
r_1^{-1/4} |\zeta |^{-1} +C_0 \epsilon r_1^{-1} |w(r_1,-\pi)|.$
\end{lem}
For the proof of this lemma, see Lemma 14.1 of \cite{EXC}, noting
that $\epsilon$-sheets satisfy the necessary hypotheses.
Using this result we bound the oscillation.
\begin{lem}
\label{oscBndPrp} Suppose $\Gamma_u$ is a 3-valued
$\epsilon$-sheet on scale 1.  Then for $\rho\geq 2$, there exists
a universal $C$ so:
\begin{equation}
\label{uThetaoscBnd}
 \osc_{C(\rho)} u_\theta \leq C\rho^{-1/4}+C\epsilon |w(\rho,-\pi)|.
\end{equation}
\end{lem}
\begin{proof}
Using Lemma \ref{exclemma} and
$u_\theta(\rho,\theta)=-\im \zeta f(\rho,\theta)$ for $\zeta=\rho
e^{i\theta}$, we compute:
\begin{eqnarray*}
\osc_{C(\rho)} u_\theta  &=& \sup_{|\zeta|=\rho} \im\left(
-c-\zeta g(\zeta)
  \right)-\inf_{|\zeta|=\rho} \im \left( -c-\zeta g(\zeta) \right) \\
                        &\leq & 2  \sup_{|\zeta|=\rho} |\zeta|
 |g(\zeta)| \; \leq \;  4 C_0 \rho^{-1/4}+2 C_0\epsilon |w(\rho/2,-\pi)|.
\end{eqnarray*}
The last inequality comes from Lemma \ref{exclemma}, setting $2
r_1=\rho$. Finally, integrate \eqref{MainInEq21} to get the bound
$|w|(\rho/2,-\pi) \leq 4|w|(\rho,-\pi)$ and choose $C$ sufficiently
large.
\end{proof}
Integrating $u_\theta$ around $ C(\rho)$ gives $w(\rho,-\pi)$,
which yields a lower bound on $\sup_{C(\rho)} u_\theta$ in terms
of the separation.  The oscillation bound of \eqref{uThetaoscBnd}
then gives a lower bound for $u_\theta$.  Indeed, for $\epsilon$
sufficiently small and large $\rho$, $u_\theta$ is positive.
\begin{prop}
\label{uThetaLowBndPrp} There exists an $\epsilon_0$ so: Suppose
 $\Gamma_u$ is a 3-valued $\epsilon$-sheet on scale 1 with
 $\epsilon<\epsilon_0$ and $w(1,\theta)\geq  C_2>0$.  Then
there exists $C_3=C_3(C_2)\geq 2$, so that on
$\pRect{C_3}{\infty}{-\pi}{\pi }$:
\begin{equation}
\label{uThetaLowBnd}
 u_\theta (\rho,\theta)\geq \frac{C_2}{8\pi}\rho^{-\epsilon}.
\end{equation}
\end{prop}
\begin{proof}
 Since $\int_{-\pi}^\pi u_\theta (\rho,\theta) \; d\theta
 =w(\rho,-\pi)$ we see
$w(\rho,-\pi)\leq 2\pi \sup_{C(\rho)} u_\theta.$ Using the
oscillation bound \eqref{uThetaoscBnd} then gives the lower bound:
\begin{equation}
\label{uThetaBnds} (1 -2\pi C\epsilon) w(\rho,-\pi) -2\pi
C\rho^{-1/4} \leq 2\pi \inf_{C(\rho)}u_\theta .
\end{equation}
Pick $\epsilon_0$ so that $2\pi C \epsilon_0\leq 1/2$. Integrating
\eqref{MainInEq21} yields $w(\rho,\theta) \geq
w(1,\theta)\rho^{-\epsilon} \geq C_2\rho^{-\epsilon}$.  Thus,
\begin{equation}
 \inf_{C(\rho)}u_{\theta} \geq
 \frac{ C_2}{4\pi} \rho^{-\epsilon} -C\rho^{-1/4}.
\end{equation}
Since $\epsilon < 1/4$, just choose $C_3$ large.
\end{proof}

\section{Decomposition of $\Sigma$}
In order to decompose $\Sigma$, we use the explicit asymptotic
properties found above to show that, away from the ``axis," $\Sigma$
consists of two strictly spiraling graphs.  In particular, this
implies that all intersections of $\Sigma$ with planes orthogonal to
the $x_3$-axis have exactly two ends.  The proof of Rado's theorem
\cite{RadoThm} then gives that $\nabla_\Sigma x_3$ is non-vanishing
and so each level set consists of one unbounded smooth curve.  A
curvature estimate and a Harnack inequality then give the lower
bound on $|\nabla_\Sigma x_3|$.

\subsection{Two Technical Lemmas}
Presently, we know that near a blow-up pair there exist two associated $4$-valued $\epsilon$-sheets on some multiple of the blow-up scale.  Moreover, the one-sided curvature estimate guarantees that for any two blow-up pairs, the part of $\Sigma$ between the associated $\epsilon$-sheets and far from the blow-up pairs consists of two multivalued graphs with good gradient bounds.  However, in order to determine the region $\mathcal{R}_S$, we must ensure that each of these ``in between" sheets is an $\epsilon$-sheet on a suitable scale.  To do so, we will need two
technical lemmas. The first gives a  bound on the number of sheets between the blow-up sheets associated to nearby blow-up pairs.  The second will imply that far enough out all of these sheets are $\epsilon$-sheets.

\begin{lem}\label{sheetbnd}
Given $K$, there is an $N$ so that: If $(y_1, s_1)$ and $(y_2,
s_2)$ are $C$ blow-up pairs of $\Sigma$ with
$y_2 \in B_{Ks_1}(y_1)$, then the number of sheets between the associated
 blow-up sheets is at most $N$.
\end{lem}
\begin{proof}  Note that for a large, universal constant $C'$ the area of $B_{C' K s_1}(y_1)\cap \Sigma$ gives a bound on $N$, so it is enough to uniformly bound this area. The chord-arc
bounds of \cite{CY} give a uniform constant $\gamma$ depending only
on $C'$ so that $B_{C'Ks_1}(y_1)\cap \Sigma$ is contained in
$\mathcal{B}_{\gamma K s_1}(y_1)$ the intrinsic ball in $\Sigma$ of
radius $\gamma K s_1$.  Furthermore, Lemma 2.26 of \cite{CY} gives a
uniform bound on the curvature of $\Sigma$ in $\mathcal{B}_{\gamma K
s_1}(y_1)$ and hence, by area comparison, a uniform bound on the area of
$\mathcal{B}_{\gamma K s_1}(y_1)$.  Since $B_{C'Ks_1}(y_1) \cap \Sigma \subset \mathcal{B}_{\gamma K s_1}(y_1)$ it
also has uniformly bounded area.
\end{proof}

We now prove the extension property.  Essentially, the existence
of a single $\epsilon$-sheet and the one-sided curvature estimate
imply that outside of a wide cone $\Sigma$ consists of the union of
weak $\epsilon$-sheets.  Thus, an initial $4$-valued $\epsilon$-sheet extends to an $N$-valued $\epsilon$-sheet on a fixed multiple of the initial scale. Results along these lines can be
found in Section 5 of \cite{MMGPD} and Section II.3 of \cite{CM3}.
\begin{lem}
\label{thm211} There exists $\epsilon_0>0$ so: Given $N>4$ and $\epsilon_0>\epsilon>0$ there exists a $\tilde{R}=\tilde{R}(\epsilon,N)>1$ so that if $\Sigma$ contains a 4-valued
$\epsilon$-sheet, $\Sigma_0$, centered at $0$ and on scale $s$ then there exists a $N$-valued $\epsilon$ sheet on scale $\tilde{R}s$, $\Sigma_1\subset \Sigma$.  Moreover, $\Sigma_1$ may be chosen so its $4$-valued middle sheet contains $\Sigma_0\backslash \set{x_1^2+x_2^2\leq \tilde{R}^2s^2}$.
\end{lem}
\begin{proof}
By rescaling we may assume that $s=1$.
For $\epsilon_0$ sufficiently small,
Corollary \ref{osccor} guarantees that the component of $\Sigma \cap
\mathbf{C}_{\epsilon}\backslash B_2$ meeting $\Sigma_0$ is the
graph of some function $u$ defined on a polar domain $\Omega_0\subset \Real^+\times \Real$ with
$|\nabla u| \leq 1$.  Moreover, this gradient bound, together with the curvature bound, implies that $|\Hess_u|\leq C_0/\rho$ (here $C_0$ is determined by the one-sided curvature estimate).  Notice that by assumption we may normalize the angular coordinate so $\pRect{2}{\infty}{-4\pi}{4\pi}\subset \Omega_0$. Set $\Omega_1 = \Phi_{u}^{-1}
(\mathbf{C}_{\epsilon/2}\backslash B_3)$. The
 distance (as subsets of $\Real^3$) between $ \mathbf{C}_{\epsilon/2} \backslash B_3$ and
$\partial \mathbf{C}_{\epsilon} \backslash B_2$ is bounded below by $\epsilon/10$. Because $\Phi_u(\Omega_0)$ is a graph with gradient bounded by $1$, the intrinsic distance (in $\Sigma$) between $\partial \Phi_u(\Omega_0)$ and $\partial \Phi_u(\Omega_1)$ is bounded below by $2d_1=\epsilon/50$.  Thus, as $\Phi_u$ is a distance non-decreasing map,
$\dist(\partial \Omega_0, \Omega_1) \geq 2d_1$.

The separation, $w$, of $u$ is defined on $\tilde{\Omega}_i=\Omega_i\cap \left(\Omega_i-(0,2\pi)\right)$ and $\dist(\partial \tilde{\Omega}_0,\tilde{\Omega}_1)\geq 2d_1$.
Because the gradient and hessian of $u$ are bounded, standard computations give that $w$ solves a uniformly elliptic and uniformly bounded second order equation in $\tilde{\Omega}_0$ (see Section 3 of \cite{MMGPD}). As $w>0$, we have a Harnack inequality, e.g. Theorems 9.20 and 9.22 of \cite{GiTr}. Thus,
for all $x \in \tilde{\Omega}_1$,
\begin{equation}\label{Harnack}
\sup_{B_{d_1}(x)} w \leq C_H \inf_{B_{d_1}(x)} w
\end{equation}
where $C_H>2$ depends only on $d_1$ and $C_0$.

Now given $\epsilon$, pick $N_0$ and $C$ from Proposition
\ref{WeakToEpsilonShtPrp}.
With $\alpha=\epsilon/C$, pick $\delta_1$ as in the remark following Corollary \ref{osccor}.  We are free to shrink $\delta_1$, and so assume $\delta_1\leq\epsilon/C$.
By \eqref{GradDecay}, there exists $R_0= R_0(\epsilon)$ such that $\pRect{R_0}{\infty}{-2\pi}{2\pi}\subset \Omega_1$ and the graph of $u$ on this polar rectangle is a 2-valued weak $\delta_1$-sheet on scale $R_0$. Thus, on
$\Omega_2= \Phi_{u}^{-1} (\mathbf{C}_{\delta_1}\backslash
B_{2R_0})$, $|\nabla u| \leq \epsilon/C$.
Thus, by Proposition \ref{WeakToEpsilonShtPrp} we need only find an $\tilde{R}>R_0$ so that $\pRect{\tilde{R}/C}{\infty}{-(N+N_0)\pi}{(N+N_0)\pi}\subset \Omega_2$. 

By hypothesis, $|w|(1,0)\leq 2\epsilon$.  Thus, integrating \eqref{MainInEq21}, gives, $|w|(\rho,0) \leq 2\epsilon\rho^\epsilon$. By increasing $R_0$, if necessary, we can assume $\Phi_u (\pRect{R_0}{\infty}{-2\pi}{2\pi}) \subset \mathbf{C}_{\delta_1/2}\backslash B_{R_0}$ -- recall the middle 4-valued sheet of $u$ is actually an $\epsilon$-sheet and thus satisfies \eqref{GradDecay}. Define $N^\pm(\rho)$ to be the number of sheets, at radius $\rho$, in $\Omega_2$ between the $\theta=0$ sheet and the top (respectively bottom) of $\Omega_2$. Note that by assumption, the $\theta=0$ sheet lies in $\mathbf{C}_{\delta_1/2}$.
We claim there exists $R_1>R_0$ so $N^\pm(\rho)\geq N_0+N$ for all $\rho
\geq R_1$.
To that end, iterated application of \eqref{Harnack}, gives
\begin{equation}\label{omegaharn}
|w|(\rho,\theta) \leq 2\epsilon \rho^\epsilon C_H^{2\pi N(\rho)/d_1}=2\epsilon \rho^\epsilon \tilde{C}_H^{N(\rho)}
\end{equation}
for $(\rho,\theta)\in\tilde{\Omega}_2$. Without loss of generality, we treat only $N(\rho)=N^+(\rho)$.  
Then,
\begin{equation*}
  \delta_1 \rho/2 \leq \sum_{k=1}^{N(\rho)} w(\rho,2\pi N(\rho)-2\pi k)
  \leq 2\epsilon \rho^\epsilon  N(\rho) \tilde{C}_H^{N(\rho)} 
  \leq  4\epsilon \rho^\epsilon   \tilde{C}_H^{2 N(\rho)} 
\end{equation*}
where the last inequality comes from the fact that $x b^x \leq
2 b^{2x}$ for $x\geq 0$ when $b>2$.  Since $\epsilon<1/2$,
$\rho^{1-\epsilon} \leq 4\tilde{C}_H^{2 N(\rho)}$.
As $\tilde{C}_H$ depends only on $\epsilon$, for $R_1=R_1(\epsilon,N)$ sufficiently large, one has for $\rho\geq R_1$ that $N(\rho)\geq N_0+N$.  Thus, set $\tilde{R}=C R_1$.
\end{proof}

\subsection{Decomposition}
To prove Theorem \ref{FirstMainThm} we first construct
$\mathcal{R}_S$.
\begin{lem}
\label{DecompFrstLem} There exist constants $C_1,R_1$ and a
sequence $(y_i,s_i)$ of $C_1$ blow-up pairs of $\Sigma$ so that:
$x_3(y_i)<x_3(y_{i+1})$ and for $i\geq 0$, $y_{i+1}\in B_{R_1 s_i}
(y_i)$ while for $i<0$, $y_{i-1}\in B_{R_1 s_i}(y_i)$.  Moreover,
if $\tilde{\mathcal{R}}_A$ is the connected component of $\bigcup_i B_{R_1
s_i}
 (y_i)\cap \Sigma$ containing $y_0$ and
$\tilde{\mathcal{R}}_S=\Sigma\backslash \tilde{\mathcal{R}}_A$, then
$\tilde{\mathcal{R}}_S$ has exactly two unbounded components, which are
(oppositely oriented) strictly spiraling,
 multivalued
graphs $\Sigma^1$ and $\Sigma^2$.  In particular,
$\nabla_\Sigma x_3 \neq 0$ on the two graphs.
\end{lem}
\begin{proof}
Fix $\epsilon<\epsilon_0$ where $\epsilon_0$ is smaller than the constant given by Theorem \ref{TwoInitShtExstThm}, 
Proposition \ref{uThetaLowBndPrp}, and Lemma \ref{thm211}.  Using this $\epsilon$, from
Theorem \ref{TwoInitShtExstThm} we obtain the blow-up constant
$C_1$ and denote by $C_2$ the lower bound on initial separation.
Suppose $0\in \Sigma$ and that $(0,1)$ is a $C_1$ blow-up pair.
From Theorem \ref{BlowUpPrsThm} there exists a constant $C_{in}$
so that there are $C_1$ blow-up pairs $(y_+,s_+)$ and $(y_-,s_-)$
with $x_3(y_-)<0<x_3(y_+)$ and $y_\pm \in B_{C_{in}}$. Note by
Lemma \ref{sheetbnd} that there is a fixed upper bound $N$
on the number of sheets between the blow-up sheets associated to
$(y_\pm,s_\pm)$ and the sheets $\Sigma^i_0$ ($i=1,2$) associated
to $(0,1)$.

By Theorem \ref{thm211}, there exists an $\tilde{R}$ so
that all the $N$ sheets above and the $N$ sheets below
$\Sigma^i_0$ are $\epsilon$-sheets centered on the $x_3$-axis on
scale $\tilde{R}$. Call these pairs of 1-valued sheets $\Sigma_j^i$, and their associated graphs $u_j^i$, with
$-N \leq j \leq N$. Integrating \eqref{MainInEq21}, we obtain from
$C_2$ and $N$ a value, $C_2'$, so that for all $\Sigma_i^j$, the
separation over $\partial D_{\tilde{R}}$ is bounded below by $C_2'$.
The non-vanishing of the right hand side of \eqref{uThetaLowBnd} is
scaling invariant, so there exists a $C_3$ such that: on each
$\Sigma_j^i$, outside of a cylinder centered on the $x_3$-axis of
radius $\tilde{R} C_3$, $(u^i_j)_\theta\neq 0$.  The chord-arc bounds of
\cite{CY} (i.e. Theorem
 0.5) then allow us to pick $R_1$ large enough so the component of
 $B_{R_1}\cap \Sigma$ containing $0$ contains this cylinder, the points
 $y_+,y_-$ and meets each $\Sigma^i_j$.
Crucially, all the statements in the theorem are
invariant under rescaling.  Thus, to finish the proof, we note that once
we find an initial blow-up pair, we can use Theorem
\ref{BlowUpPrsThm} to iteratively construct a sequence of $C_1$
blow-up pairs $(y_k,s_k)$ and choose $R_1$ as described above. As
$\Sigma$ is not flat, there is an $r_0$ so that
$\sup_{B_{r_0} \cap \Sigma}|A|^2 \geq 16 C_1^2 r_0^{-2}.$ By Lemma
\ref{CM251Lem}, this gives the existence of an initial blow-up pair.
\end{proof}
With the given decomposition, we now show that each level set of
$x_3$, outside of a large ball,
consists of exactly two proper curves.  
\begin{lem} \label{RadoPrepLem}
For all $h$, there exist $\alpha,\rho_0>0$ so that for all
 $\rho>\rho_0$ the
set
 $\Sigma\cap \set{x_3=c}\cap \set{x_1^2+x_2^2 =\rho^2}$ consists of
 exactly two points for $|c-h|\leq \alpha$.
\end{lem}
\begin{proof}
First note, for $\rho_0$ large, the intersection is never empty by
the strong half-space theorem \cite{HoffmanMeeksHalfSpace}, the
 maximum principle
and because $\Sigma$ is properly embedded.  Without loss of generality we may
assume $h=0$ with $0\in Z^0=\Sigma\cap\set{x_3=0}$ and $|A|^2(0)\neq
0$.  Let $R_1$ and the set of blow-up pairs be given by Lemma
\ref{DecompFrstLem} and $\Sigma^i$ the unbounded components of $\tilde{\mathcal{R}}_S$. There then exists $\rho_0$ so for $2
\rho>\rho_0$, $\set{x_1^2+x_2^2 =\rho^2}\cap Z^0$ lies in the set
$\Sigma^1\cup \Sigma^2$. If no such $\rho_0$ existed then, since the blow-up
pairs lie within a cone, there would exist $\delta>0$ and a subset
of the blow-up pairs $(y_i,s_i)$ so $0\in  B_{\delta R_1
s_{i}}(y_{i})$. However, Lemma 2.26 of \cite{CY}, with $K_1=\delta
R_1$, would then imply $|A|^2(0)\leq K_2 s_{i}^{-2}$ for all $i$,
i.e. $|A|^2(0)=0$, a contradiction.  Now, for some small $\alpha$
and $\rho>\rho_0$, $Z^c\cap \set{x_1^2+x_2^2=\rho^2}$ lies in
$\Sigma^1\cup \Sigma^2$ for all $|c|<\alpha$, and so
$\set{x_1^2+x_2^2=\rho^2}\cap \set{-\alpha<x_3<\alpha}\cap \Sigma$
consists of the union of the graphs of $u^1$ and $u^2$ over the
circle $\partial D_\rho$, both of which are monotone increasing in
height.
\end{proof}
As $x_3$ is harmonic on $\Sigma$, Proposition \ref{SecondMainThm}
is an immediate
consequence of the previous result and the proof of Rado's theorem.
Recall Rado's theorem \cite{RadoThm} implies that any minimal
surface whose boundary is a graph over the boundary of a convex
domain is a graph over that domain.  The proof of this reduces to
showing that a non-constant harmonic function on a closed disk has
an interior critical point if and only if the level curve of the
function through that point meets the boundary in at least 4 points,
which is exactly what we use.  
We now show Theorem
\ref{FirstMainThm}:
\begin{proof}
First of all, set $\mathcal{R}_S = \Sigma^1 \cup \Sigma^2$, the infinite-valued graphs defined by Lemma \ref{DecompFrstLem}. 
It will then suffice to adjoin the bounded components of $\tilde{\mathcal{R}}_S$ to $\tilde{\mathcal{R}}_A$ and to show that $|\nabla_\Sigma
x_3|$ is bounded below on the new $\mathcal{R}_A$. Suppose that $(0,1)$ is a
blow-up pair. By the chord-arc bounds of \cite{CY}, there exists
$\gamma$ large enough so that the intrinsic ball of radius $\gamma
R_1$ contains $\Sigma \cap B_{R_1}$.  Lemma 2.26 of \cite{CY}
implies that curvature is bounded in $ B_{2\gamma R_1} \cap \Sigma$
by some $K=K(\gamma R_1)$.  The function $v=-2 \log |\nabla_\Sigma
x_3| \geq 0$ is well defined and smooth by Proposition
\ref{SecondMainThm} and standard computations give $\Delta_\Sigma
v=|A|^2$.  Then, since $|\nabla_\Sigma x_3|=1$ somewhere in the
component of $B_1 (0)\cap \Sigma$ containing 0, we can apply a
Harnack inequality (e.g. Theorems 9.20 and 9.22 of \cite{GiTr}) to
obtain an upper bound for $v$ on the intrinsic ball of radius
$\gamma R_1$ that depends only on
 $K$ and $\gamma R_1$. Consequently,
there is a lower bound $\epsilon_0$ on $|\nabla_\Sigma x_3|$ in
$\Sigma \cap B_{R_1}$.  Since this bound is scaling invariant, the
same bound holds around any blow-up pair.  Finally, any bounded
 component, $\Omega$, of $\tilde{\mathcal{R}}_S$ has boundary in $\tilde{\mathcal{R}}_A$ and
 so, since $v$ is subharmonic, $|\nabla_\Sigma x_3|\geq \epsilon_0$ on
 $\Omega$.  Thus, by adjoining all such bounded $\Omega$ to $\tilde{\mathcal{R}}_A$
 we obtain Theorem \ref{FirstMainThm}.
\end{proof}
\section{Concluding Uniqueness}
Since $\nabla_\Sigma x_3$ is non-vanishing and the level sets of
$x_3$ in $\Sigma$ consist of a single curve, the map
$z=x_3+ix_3^*:\Sigma \to \mathbb{C}$ is a global holomorphic
coordinate (here $x_3^*$ is the harmonic conjugate of $x_3$).
Additionally, $\nabla_\Sigma x_3\neq 0$ implies that the normal of
$\Sigma$ avoids $(0,0,\pm 1)$.  Thus, the stereographic projection
of the Gauss map, denoted by $g$, is a holomorphic map $g:\Sigma \to
\mathbb{C}\backslash \set{0}$.  By monodromy, there exists a
holomorphic map $h=h_1+ih_2:\Sigma \to \mathbb{C}$ so that $g=e^h$.
We will use $h$ to show that $z$ is actually a conformal
diffeomorphism between $\Sigma$ and $\mathbb{C}$.  As the same is
then true for $h$, embeddedness and the Weierstrass representation
imply $\Sigma$ is the helicoid.
\subsection{Structure of $h$}
As $\nabla_\Sigma x_3$ is the projection of $\mathbf{e}_3$ onto $T\Sigma$, one can compute $|\nabla_\Sigma x_3|$ by comparing it to the projection of $\mathbf{e}_3$ onto the unit normal of $\Sigma$.  This gives the following relation between $\nabla_\Sigma x_3$, $g$
and $h$:
\begin{equation}
\label{gradx3h} |\nabla_\Sigma x_3| =2 \frac{ |g|}{1+|g|^2}\leq
2e^{-|h_1|}.
\end{equation} An immediate consequence of \eqref{gradx3h} and the
decomposition of Theorem \ref{FirstMainThm} is that there exists
$\gamma_0>0$ so on $\mathcal{R}_A$, $|h_1(z)|\leq \gamma_0$. This
imposes strong rigidity on $h$:
\begin{prop}
\label{hpoly} Let $\Omega_\pm=\set{x\in \Sigma: \pm h_1(x)\geq
2\gamma_0}$ then $h$
 is a proper conformal diffeomorphism from $\Omega_\pm$ onto the closed
 half-spaces $\set{z:\pm\re z\geq 2\gamma_0}$.
\end{prop}
\begin{proof} Let $\gamma>\gamma_0$ be a regular value
of $h_1$.  Such $\gamma$ exist by Sard's theorem and indeed form a
dense subset of $(\gamma_0,\infty)$.  We claim that the set of smooth
curves $Z=h_1^{-1}(\gamma)$ has exactly one component.
 Note that $Z$ is non-empty by \eqref{GradDecay} and \eqref{gradx3h}.  By
construction, $Z$ is a subset of $\mathcal{R}_S$ and, up to choosing
an orientation, $Z$ lies in the graph of $u^1$, which we will
henceforth denote as $u$. Let us parametrize one of the
components of $Z$ by $\phi(t)$, non-compact by the maximum
principle, and write $\phi(t)=\Phi_u (\rho(t),\theta(t))$.  Note, $h_2(\phi(t))$ is monotone in $t$ by the Cauchy-Riemann equations and because $\gamma_0$ is a regular value of $h_1$ and so $\nabla_\Sigma h_1(\phi(t))\neq 0$.

At the point $\Phi_u(\rho,\theta)$ we compute:
\begin{equation} \label{gInTermUeqn}
g(\rho,\theta)=-\frac{1}{\sqrt{1+|\nabla u|^2}-1}
\left(u_\rho(\rho,\theta)+i\frac{u_\theta
(\rho,\theta)}{\rho}\right)e^{i\theta}.
\end{equation}
Since $u_\theta(\rho(t),\theta(t))>0$, there exists a function
$\tilde{\theta}(t)$ with $\pi <\tilde{\theta}(t)< 2\pi$ such that
\begin{equation}
|\nabla u|(\rho(t),\theta(t)) e^{i\tilde{\theta}(t)}=-u_\rho
 (\rho(t),\theta(t))
-i\frac{u_\theta(\rho(t),\theta(t))}{\rho(t)} .
\end{equation}
Thus, $h_2(\phi(t))=\theta(t)+\tilde{\theta}(t)$.

We now claim that, up to replacing $\phi(t)$ by $\phi(-t)$,
$\lim_{t\to \pm \infty} h_2 (\phi(t))=\pm\infty$. Without loss of
generality, we need only rule out $\lim_{t\to \infty}
h_2(\phi(t))=R<\infty$.  Suppose this occurred, then by the
monotonicity of $h_2$, $h_2(\phi(t))<R$.  The formula for
$h_2(\phi(t))$ implies that, for $t$ large, $\phi(t)$ lies in one
sheet.   The decay estimates \eqref{GradDecay} together with
\eqref{gradx3h} imply $\rho(t)$ cannot became arbitrarily large and
so the positive end of $\phi$ lies in a compact set.  Thus, there is
a sequence of points $p_j = \phi(t_j)$, with $t_j$ monotonically
increasing to $\infty$, so $p_j\to p_\infty \in \Sigma$.  By the
continuity of $h_1$,
 $p_\infty \in Z$, and since $h_2(p_j)$ is monotone increasing with
 supremum $R$, $h_2(p_\infty)=R$,
and so $p_\infty$ is not in $\phi$.  However, $p_\infty\in Z$
implies
 $h'(p_\infty) \neq 0$ and so $h$ restricted to a small neighborhood of
$p_\infty$ is a diffeomorphism onto its image, contradicting
$\phi$ coming arbitrarily close to $p_\infty$.

Thus, the formula for $h_2(\phi(t))$, its monotonicity, and the
bound on $\tilde{\theta}$ show that $\theta(t)$ must extend from
$-\infty$ to $\infty$.  We now conclude that there are at most a
finite number of components of $Z$.  Namely, since $\theta(t)$ runs
from $-\infty$ to $\infty$ we see that every component of $Z$ must
meet the ray $\eta(\rho)=\Phi_u(\rho,0), \rho \in (\rho_0,\infty)$,
where $\rho_0$ is the smallest value so $\eta\subset \mathcal{R}_S$.
Again, the gradient decay of \eqref{GradDecay}  says that the set of
intersections of $Z$ with $\eta$ lies in a compact set, and so
consists of a finite number of points.  Now, suppose there was more
than one component of $Z$.  Looking at the intersection of $Z$ with
$\eta$, we order these components innermost to outermost;
parametrize the innermost curve by $\phi_1(t)$ and the outermost by
$\phi_2(t)$. Pick $\tau$ a regular value for $h_2$, and parametrize
the component of $h_2^{-1}(\tau)$ that meets $\phi_1$ by
$\sigma(t)$, writing $\sigma(t)=\Phi_u(\rho(t),\theta(t))$ in
$\mathcal{R}_S$. From the formula for $h_2$, $|\theta(t)-\tau|\leq
2\pi$. Again, $\sigma(t)$ cannot have an end in a compact set, so
$\rho(t) \to \infty$. Hence, $\sigma$ must also intersect $\phi_2$
contradicting the monotonicity of $h_1$ on $\sigma$.

Thus, when $\gamma>\gamma_0$ is a regular value of $h_1$,
$h_1^{-1}(\gamma)$ is a single smooth curve.  We claim this implies
that all $\gamma>\gamma_0$ are regular values.  Suppose
$\gamma'>\gamma_0$ were a critical value of $h_1$. However, as $h_1$
is harmonic, the proof of Rado's theorem implies for
$\gamma>\gamma_0$, a regular value of $h_1$ near $\gamma'$,
$h_1^{-1}(\gamma)$ would have at least two components.  Thus,
$h:\Omega_+\to \set{z:\re z\geq 2\gamma_0}$ is a conformal
diffeomorphism that maps boundaries onto boundaries, immediately
implying that $h$ is proper on $\Omega_+$.  An identical argument applies to
 $\Omega_-$.
\end{proof}
By looking at $z$, which already has well understood behavior away
from $\infty$, we see that $\Sigma$ is conformal to $\mathbb{C}$
with $z$ providing an identification.
\begin{prop} \label{ConfDiffprop} The map $h\circ z^{-1}:\mathbb{C}\to
 \mathbb{C}$ is linear.
\end{prop}
\begin{proof}
We first show that $z$ is a conformal diffeomorphism between
$\Sigma$
 and $\mathbb{C}$, i.e. $z$ is onto. This will follow if we show $x_3^*$
 goes from $-\infty$ to
$\infty$ on the level sets of $x_3$.  The key fact is: each level
set of $x_3$ has one end in $\Omega_+$ and the other in $\Omega_-$.
This is an immediate consequence of the radial decay along level
curves of $x_3$, which follows from the one-sided curvature
estimate. Indeed, for any $\epsilon>0$ there is a $\delta_\epsilon$
so that if $\mathbf{C}_{\delta_\epsilon}$ contains a weak 2-valued
$\delta_\epsilon$ sheet on scale $1$ then all components of $\Sigma
\cap \left(\mathbf{C}_{\delta_\epsilon}\backslash B_2\right)$ can be
expressed as graphs with gradient bounded by $\epsilon$.  By the
faster than linear gradient decay of \eqref{GradDecay} and a
rescaling, such a weak sheet can always be found. Every
level set of $x_3$ must lie in this set, and so far enough out, each point of
the level set lies on a graph with gradient bounded by
$\epsilon$.  However, \eqref{gInTermUeqn} implies that at such points $|h_1|\geq -1/4 \ln \epsilon$, forcing $x_3$ to run from
$-\infty$ to $\infty$ along the curve $\partial \Omega_+$. Thus, 
$z(\partial \Omega_+)$ splits $\mathbb{C}$ into two components with
only one, $V$, meeting $z(\Omega_+)=U$. After conformally
straightening the boundary of $V$ (using the Riemann mapping
theorem) and precomposing with $h|_{\Omega_+}^{-1}$, we obtain a map
from a closed half-space \emph{into} a closed half-space with the
boundary mapped \emph{into} the boundary. We claim
 that this map is necessarily \emph{onto}, that is $U$ equals $\bar{V}$. Suppose
 it was not onto, then a
Schwarz reflection would give a holomorphic map from $\mathbb{C}$
into
 a simply connected proper subset of $\mathbb{C}$.  Because the latter
 is conformally a disk, Liouville's theorem would imply this map was
 constant, a contradiction.  As a consequence,
 if $p\to \infty$ in $\Omega_+$ then $z(p)\to \infty$, with the same
 true in
 $\Omega_-$.
Thus, along each level set of $x_3$, $|x_3^*(p)|\to \infty$ and so
$z$ is onto.  Then, by the level set analysis in the proof of Proposition
 \ref{hpoly} and Picard's theorem, $h\circ z^{-1}$ is a polynomial and is indeed
 linear.
\end{proof}

\subsection{Concluding Uniqueness}
After a translation in $\Real^3$ and a rebasing of $x_3^*$,
$h(p)=\alpha z(p)$ for some $p \in \Sigma$.  As $dz$ is the height
differential, the Weierstrass representation gives, on the
curve parameterized by $z=0+it$, that
 \begin{equation*}
x_1(it)=|\alpha|^{-2}\left(\alpha_2 \sinh(\alpha_2 t) \sin (\alpha_1 t)-\alpha_1
 \cosh(\alpha_2 t)\cos(\alpha_1 t)\right) 
\end{equation*} and
\begin{equation*}
 x_2(it)=|\alpha|^{-2}\left(\alpha_2\sinh(\alpha_2 t)\cos (\alpha_1 t)+\alpha_1\cosh(\alpha_2
 t)\sin(\alpha_1 t)\right)
\end{equation*}
where $\alpha=\alpha_1+i\alpha_2$.  By
 inspection, this curve is only embedded when $\alpha_1=0$, i.e. if
 $\alpha=i\alpha_2$.  The factor $\alpha_2$ corresponds to a homothetic
rescaling and so $\Sigma$ is the helicoid.

\section{Local Result}
Consider two oriented surfaces $\Sigma_1, \Sigma_2\subset \Real^3$, so that
$\Sigma_2$ is the graph of $\nu$ over $\Sigma_1$.  Then the map
$\phi:\Sigma_1\to \Sigma_2$ defined as $\phi(x)=x+\nu(x)
\mathbf{n}(x)$ is smooth.  Moreover, if $\nu$ is small in a $C^1$
sense, $\phi$ is an ``almost isometry".
\begin{lem}
\label{GraphLem} Let $\Sigma_2$ be the graph of $\nu$ over
$\Sigma_1$, with $\Sigma_1\subset B_{R}$, $\partial \Sigma_1
\subset
\partial B_{R}$ and $|A_{\Sigma_1}|\leq 1$.  Then, for $\epsilon$
 sufficiently small,
$|\nu|+ |\nabla_{\Sigma_1} \nu|\leq \epsilon$ implies $\phi$
 is a diffeomorphism with $1-\epsilon\leq ||d\phi||\leq 1+\epsilon$.
\end{lem}
\begin{proof}
For $\epsilon$ sufficiently small (depending on $\Sigma_1$),
$\phi$ is injective. Working in $\Real^3$, given orthonormal
vectors $e_1,e_2\in T_p \Sigma_1$ we compute:
\begin{equation}
d\phi_p (e_i)=e_i+\langle \nabla_{\Sigma_1} \nu (p), e_i\rangle
\mathbf{n}(p)+\nu(p) D\mathbf{n}_p (e_i).
\end{equation}
The last two terms are together controlled by $\epsilon$. Hence,
$1-\epsilon<|d\phi_p (e_i)|<1+\epsilon$.
\end{proof}

\begin{proof} (of Theorem \ref{FifthMainThm})
By rescaling we may assume that $s=1$. We proceed by
contradiction.
  Suppose no such $R'$ existed for fixed $\epsilon, R$.  That is, there
 exists a sequence of counter-examples; $\Sigma'_i\in \mathcal{E}(1,0, R_i)$,  $(0,1)$ a $C$
blow-up pair of each $\Sigma_i'$ and $R\leq R_i\to \infty$, but
$\Sigma_i$, the component of $B_R\cap \Sigma_i'$ containing zero, is
not close to a helicoid.

By definition, $|A_{\Sigma'_i} (0)|^2=C>0$ for all $\Sigma_i'$ and
so the lamination theory of Colding and Minicozzi implies that a
subsequence of the $\Sigma'_i$ converge smoothly and with
multiplicity one to $\Sigma_\infty$, a complete embedded minimal
disk.  Namely, in any ball centered at $0$ the curvature of
$\Sigma_i$ is uniformly bounded by Lemma 2.26 of \cite{CY}.
Furthermore, the chord-arc bounds of \cite{CY} give uniform area
bounds and so by standard compactness arguments one has smooth
convergence (possibly with multiplicity) to $\Sigma_\infty$.  If the
multiplicity of the convergence is greater than 1, then one can
construct a positive solution to the Jacobi equation (see Appendix B
of \cite{CM5}). That implies $\Sigma_\infty$ is stable, and thus a
plane by Schoen's extension of the Bernstein theorem \cite{SchoenStab}, contradicting the curvature at $0$.
Thus, as $\Sigma_\infty \in \mathcal{E}(1,0)$ is non-flat, Theorem
\ref{FourthMainThm} implies it is a helicoid.  We may,
by rescaling, assume $\Sigma_\infty$ has curvature $1$ along its
axis.

For any fixed $R'$ a subsequence of $\Sigma'_{i}\cap B_{R'}$
converges to $\Sigma_\infty\cap B_{R'}$ in the smooth topology.
And so, for any $\epsilon$, with $i$ sufficiently large, we find a
smooth $\nu_i$ defined on a subset of $\Sigma_\infty$ so that
$|\nu_i|+|\nabla_{\Sigma_\infty} \nu_i|<\epsilon$ and the graph of
$\nu_i$ is $\Sigma'_i\cap B_{R'}$. Choosing $R'$ large enough to
ensure minimizing geodesics between points in $\Sigma_i$ lie in
 $\Sigma_i'\cap B_{R'}$ (using the
chord-arc bounds of \cite{CY}), Lemma \ref{GraphLem} gives the
desired contradiction.
\end{proof}
\appendix
\section{Structural Results of Colding and
Minicozzi}\label{structureapp}

For the convenience of the reader, we gather here some of the results of Colding and Minicozzi on the structure of embedded minimal surfaces.   The foundation of their work is their
description of embedded minimal disks in \cite{CM1,CM2,CM3,CM4},
which underpins their results for more general topologies in
\cite{CM5}. Roughly speaking, they show that embedded minimal
disks (with boundary lying on the boundary of a ball) fall into precisely two classes. On the one hand, if the curvature of such a surface,
$\Sigma$, is everywhere small, then the surface is nearly flat and hence
modeled on a plane (i.e. is a single-valued graph). On the other
hand, when $\Sigma$ has (far from the boundary) a point with large curvature then
it is modeled on a helicoid. That is, in a smaller ball $\Sigma$
consists of two multivalued graphs that spiral together and that
are glued along an ``axis" of large curvature.  In proving such a
qualitative description, Colding and Minicozzi show many
quantitative results about the behavior of the sheets and of the
axis.  It is these later results that we use and describe in more
detail below.

We point out that their work is interior theory, that is it holds far from the boundary.  In our application, the minimal disks are complete and without boundary which simplifies things
somewhat and so the reader may wish to assume this and ignore the conditions regarding the boundary.
We have also, where needed, changed notation to that used in the present paper.
\subsection{Existence of Multivalued Graphs}
Definition \ref{defbup} gives a condition that specifies the
points and scales of an embedded minimal disk of large curvature
that are of particular interest in the theory -- recall we refer
to these pairs as blow-up pairs. These are points of almost
maximal curvature in a ball with scale which is inversely
proportional to the curvature at the point. A good example
of such blow-up pairs is provided by points on the axis of the
helicoid, as there the scale is proportional to the separation
between the sheets of the helicoid. This is the model
behavior for any embedded minimal disk containing a blow-up pair
-- i.e. in any embedded minimal disk, near a blow-up pair one has
a small multivalued graph forming on the scale of the pair.
 \begin{thm} \label{CM204Thm}(Theorem 0.4 of \cite{CM2})
 Given $N, \omega>1$ and $\epsilon> 0$, there exists $C = C (N, \omega,\epsilon)> 0$ so: Let
$0 \in \Sigma\in \mathcal{E}(1,0,R)$. If $(0,r_0)$ is a $C$
blow-up pair for $0 < r_0 < R$, then there exist $\bar{R} < r_0
/\omega$ and (after a rotation) an $N$-valued graph $\Sigma_g
\subset\Sigma$ over $D_{\omega\bar{R} }\backslash D_{\bar{R}}$
with gradient $\leq \epsilon$, and $dist_\Sigma(0, \Sigma_g ) \leq
\bar{R}$.
\end{thm}

\subsection{Extending the Graphs}
Using the initial small multivalued graph, Colding and Minicozzi
show that it can be extended, as a graph and within the surface
$\Sigma$, nearly all the way to the boundary of $\Sigma$.
\begin{thm}\label{CM103Thm}(Theorem
0.3 of \cite{CM1}) Given $\tau>0$ there exist $N,\Omega,\epsilon>0$
so that the following hold:  Let $\Sigma\in \mathcal{E}(1,0,R_0)$.
If $\Omega r_0<1<R_0/ \Omega$ and $\Sigma$ contains a $N$-valued
graph $\Sigma_g$ over $D_1\backslash D_{r_0}$ with gradient $\leq
\epsilon$ and
$\Sigma_g \subset \mathbf{C}_\epsilon$
then $\Sigma$ contains a $2$-valued graph $\Sigma_d$ over
$D_{R_0/\Omega} \backslash D_{r_0}$ with gradient $\leq \tau$
and $(\Sigma_g)^M\subset\Sigma_d$.
\end{thm}
Here $(\Sigma_g)^M$ indicates the ``middle" 2-valued sheet of
$\Sigma_g$. Combining this with Theorem \ref{CM204Thm}, gives the existence of a multivalued graph near a
blow-up pair that extends almost all the way to the boundary. Namely,
Theorem 0.2 of \cite{CM2}:
\begin{thm}\label{CM202Thm}
 Given $N\in \mathbb{Z}^+, \epsilon>0$, there exist $C_1 , C_2,C_3 > 0$ so: Let $0 \in \Sigma
\in \mathcal{E}(1,0,R)$. If $(0,r_0)$ is a $C_1$ blow-up pair then
there exists (after a rotation) an $N$ -valued graph
$\Sigma_g\subset\Sigma$ over $D_{R/C_3} \backslash D_{2r_0}$ with
gradient $\leq \epsilon$ and $\Sigma\subset\mathbf{C}_\epsilon.$
Moreover, the separation of $\Sigma_g$ over $\partial D_{r_0}$ is
bounded below by $C_2 r_0$.
\end{thm}
Note that the lower bound on the initial separation is not
explicitly stated in Theorem 0.2 of \cite{CM2} but is proved in
Proposition 4.15 of \cite{CM2}.

\subsection{The Second Multivalued Graph} Colding and Minicozzi
show that, ``between the sheets" of $\Sigma_g$, $\Sigma$ consists of
exactly one other multivalued graph. That is we have at least that
part of $\Sigma$ looks like (a few sheets of) a helicoid. Precisely,
one has Theorem I.0.10 of \cite{CM4}:
\begin{thm}\label{CM4I010Thm}
Suppose $0\in\Sigma\in \mathcal{E}(1,0,4R)$ and $\Sigma_1 \subset
\mathbf{C}_1 \cap \Sigma$ is an $(N + 2)$-valued graph of $u_1$ over
$ D_{2R} \backslash D_{r_1}$ with $|\nabla u_1| \leq \epsilon$ and
$N \geq 6$. There exist $C_0>2$ and $\epsilon_0> 0$ so that if $R
\geq C_0 r_1$ and $\epsilon_0 \geq \epsilon$, then $E\cap
\Sigma\backslash\Sigma_1$ is an (oppositely oriented) $N $-valued
graph $\Sigma_2$.
\end{thm}
Here $E$ is the region between the sheets of $\Sigma_1$ and is the same as \eqref{Edefn}:
\begin{multline} \label{CM4ERegionEqn}
\{(r\cos \theta, r\sin \theta, z): 2r_1<r<R, -2\pi <\theta <0, \\
u_1 (r, \theta-N\pi) < z < u_1 (r, \theta  +N \pi)\}.
\end{multline}
By Theorem \ref{CM4I010Thm}, near a blow-up point there are two
multivalued graphs that spiral together and extend within $\Sigma$
almost all the way to the boundary of $\Sigma$.

\subsection{Finding Blow-up Pairs} The existence of two multivalued
graphs spiraling together allows Colding and Minicozzi to use the
following result from \cite{CM3} in order to show there are
regions of large curvature above and below the original blow-up
pairs. This is Corollary III.3.5 of \cite{CM3}:
\begin{cor}\label{CM3III35Cor} Given $C_1$ there exists $C_2$ so: Let $0\in\Sigma \in \mathcal{E}(1,0,2 C_2 r_0)$.
 Suppose $\Sigma_1,\Sigma_2\subset \Sigma\cap \mathbf{C}_1$ are graphs of $u_i$ satisfying \eqref{MainInEq21} on $\pRect{r_0}{C_2 r_0}{-2\pi}{2\pi}$, $u_1(r_0, 2\pi)<u_2(r_0,0)<u_1(r_0,0)$,
 and $\nu\subset \partial \Sigma_{0,2r_0}$ a curve from $\Sigma_1$ to $\Sigma_2$. (Here, $\Sigma_{0,2r_0}$ denotes the component of $\Sigma \cap B_{2r_0}$ containing $0$.) Let $\Sigma_0$ be the component of $\Sigma_{0,C_2
r_0}\backslash (\Sigma_1\cup \Sigma_2 \cup \nu)$ which does not
contain $\Sigma_{0,r_0}$.  Then
\begin{equation}
\sup_{x\in \Sigma_0\backslash B_{4r_0}} |x|^2 |A|^2 (x)\geq
4C_1^2.
\end{equation}
\end{cor}
By a standard blow-up argument if there is large curvature in a
ball (measured with respect to the scale of the ball) then there
exists a blow-up pair in the ball.  This is Lemma 5.1 of
\cite{CM2}:
 \begin{lem} \label{CM251Lem}
 If $0\in \Sigma\in\mathcal{E}(1,0,r_0)$ and $\sup_{B_{r_0/2}\cap \Sigma} |A|^2 \geq 16 C^2r_0^{-2}$ then there exists a pair $(y, r_1)$ with $y\in \Sigma$ and $r_1<r_0-|y|$ so $(y,r_1)$ is a $C$ blow-up pair.
 \end{lem}
Combining Theorem \ref{CM204Thm}, Corollary \ref{CM3III35Cor}, and
Lemma \ref{CM251Lem}, gives the existence of blow-up pairs above
and below an initial pair.

\subsection{The One-sided Curvature Estimate}\label{oscsec} Using the above results, Colding and Minicozzi prove that an embedded minimal disk
that is close to and on one side of a plane has uniformly bounded
curvature.  The precise statement of the
one-sided curvature estimate is the following:
\begin{thm}\label{oscthm}(Theorem 0.2 \cite{CM4})
There exists $\epsilon>0$ so that if $\Sigma \subset B_{2r_0} \cap
\{x_3>0\} \subset \Real^3$ is an embedded minimal disk with
$\partial \Sigma \subset \partial B_{2r_0}$, then for all
components, $\Sigma'$ of $\Sigma \cap B_{r_0}$ which intersect
$B_{\epsilon r_0}$, we have
\begin{equation}
\sup_{\Sigma'} |A_{\Sigma}|^2 \leq r_0^{-2}.
\end{equation}
\end{thm}
Rescaled catenoids show that the surface must be an embedded
disk.  As a consequence of this estimate, if an embedded
minimal disk, $\Sigma$, contains a two-valued graph lying outside of a cone union a ball, then all components of $\Sigma$ in the complement
of a larger cone (and larger ball) are multivalued graphs.
The nearly flat two-valued graph takes the place of the
plane in Theorem \ref{oscthm}. Precisely,
\begin{cor}\label{osccor}(Corollary I.1.9 in \cite{CM4}) There exists $\delta_0$
so that the following holds:  Let $\Sigma \in
\mathcal{E}(1,0,2R)$.  If $\Sigma$ contains a 2-valued graph
$\Sigma_d \subset \mathbf{C}_{\delta_0}$ over $D_R\backslash
D_{r_0}$ with gradient $\leq \delta_0$, then each component of
\[(\mathbf{C}_{\delta_0} \cup B_{R/2}) \cap \Sigma \backslash
B_{2r_0}\] is a multivalued graph with gradient $\leq 1$.
\end{cor}
\begin{rem}
Schauder estimates imply that, by shrinking $\delta_0$, one has $|\nabla u|\leq \alpha$.
\end{rem}

\bibliographystyle{amsplain}
\bibliography{FinalDraft}

\providecommand{\bysame}{\leavevmode\hbox to3em{\hrulefill}\thinspace}
\providecommand{\MR}{\relax\ifhmode\unskip\space\fi MR }
\providecommand{\MRhref}[2]{%
  \href{http://www.ams.org/mathscinet-getitem?mr=#1}{#2}
}
\providecommand{\href}[2]{#2}
\begin{thebibliography}{10}

\bibitem{BB3}
J.~Bernstein and C.~Breiner, \emph{Conformal structure of minimal surfaces with
  finite topology}, To appear. Comm. Math. Helv. {\tt
  http://arxiv.org/abs/0810.4478}.

\bibitem{BB2}
\bysame, \emph{{Distortions of the Helicoid}}, {Geometriae Dedicata.}
  \textbf{137} (2008), no.~1, 143--147.

\bibitem{CM5}
T.~H. Colding and W.~P.~Minicozzi II, \emph{{The Space of Embedded Minimal
  Surfaces of Fixed Genus in a 3-manifold V; Fixed Genus}}, Preprint.

\bibitem{MMGPD}
\bysame, \emph{Multivalued minimal graphs and properness of disks}, Int. Math.
  Res. Not. (2002), no.~21, 1111--1127.

\bibitem{CMPVNP}
\bysame, \emph{{Embedded minimal disks: Proper versus nonproper-Global versus
  local}}, Trans. AMS \textbf{356} (2003), no.~1, 283--289.

\bibitem{EXC}
\bysame, \emph{An excursion into geometric analysis}, Surv. in Diff. Geom.
  \textbf{IX} (2004), 83--146.

\bibitem{CM1}
\bysame, \emph{{The space of embedded minimal surfaces of fixed genus in a
  3-manifold I; Estimates off the axis for disks}}, Ann. of Math. (2)
  \textbf{160} (2004), no.~1, 27--68.

\bibitem{CM2}
\bysame, \emph{{The space of embedded minimal surfaces of fixed genus in a
  3-manifold II; Multi-valued graphs in disks}}, Ann. of Math. (2) \textbf{160}
  (2004), no.~1, 69--92.

\bibitem{CM3}
\bysame, \emph{{The space of embedded minimal surfaces of fixed genus in a
  3-manifold III; Planar domains}}, Ann. of Math. (2) \textbf{160} (2004),
  no.~2, 523--572.

\bibitem{CM4}
\bysame, \emph{{The space of embedded minimal surfaces of fixed genus in a
  3-manifold IV; Locally simply connected}}, Ann. of Math. (2) \textbf{160}
  (2004), no.~2, 573--615.

\bibitem{CMEMD}
\bysame, \emph{Embedded minimal disks}, Global theory of minimal surfaces, Clay
  Math. Proc., vol.~2, Amer. Math. Soc., Providence, RI, 2005, pp.~405--438.

\bibitem{CMPNAS}
\bysame, \emph{Shapes of embedded minimal surfaces}, PNAS \textbf{103} (2006),
  no.~30, 11106--11111.

\bibitem{CY}
\bysame, \emph{{The Calabi-Yau conjectures for embedded surfaces}}, Ann. of
  Math. \textbf{167} (2008), no.~1, 211--243.

\bibitem{CKMR}
P.~Collin, R.~Kusner, W.~H.~Meeks III, and H.~Rosenberg, \emph{The topology,
  geometry and conformal structure of properly embedded minimal surfaces}, J.
  Differential Geom. \textbf{67} (2004), no.~2, 377--393.

\bibitem{GiTr}
D.~Gilbarg and N.~S. Trudinger, \emph{Elliptic partial differential equations
  of second order}, Springer-Verlag, 1998.

\bibitem{HoffmanMeeksHalfSpace}
D.~Hoffman and W.~H.~Meeks III, \emph{{The strong halfspace theorem for minimal
  surfaces}}, Inventiones Mathematicae \textbf{101} (1990), no.~1, 373--377.

\bibitem{HW3}
D.~Hoffman and B.~White, \emph{{Sequences of embedded minimal disks whose
  curvatures blow up on a prescribed subset of a line}}, Preprint. {\tt
  http://arxiv.org/pdf/0905.0851}.

\bibitem{HW}
\bysame, \emph{Genus-one helicoids from a variational point of view}, Comment.
  Math. Helv. \textbf{83} (2008), no.~4, 767--813.

\bibitem{MR}
W.~H.~Meeks III and H.~Rosenberg, \emph{The uniqueness of the helicoid}, Ann.
  of Math. (2) \textbf{161} (2005), no.~2, 727--758.

\bibitem{K}
S.~Khan, \emph{{A minimal lamination of the unit ball with singularities along
  a line segment}}, Illinois J. Math. \textbf{53} (2009), no.~3, 833--855.

\bibitem{Kl}
S.~Kleene, \emph{{A minimal lamination with Cantor set-like singularities}},
  Preprint. {\tt http://arxiv.org/pdf/0910.0199}.

\bibitem{M1}
William~H. Meeks, III, \emph{Regularity of the singular set in the
  {C}olding-{M}inicozzi lamination theorem}, Duke Math. J. \textbf{123} (2004),
  no.~2, 329--334.

\bibitem{RadoThm}
T.~Rado, \emph{{On the problem of Plateau}}, Ergebnisse der Math. und ihrer
  Grenzgebiete \textbf{2} (1953).

\bibitem{SchoenStab}
R.~Schoen, \emph{Seminar on minimal submanifolds}, Ann. of Math. Studies, vol.
  103, ch.~Estimates for stable minimal surfaces in three-dimensional
  manifolds, pp.~111--126, Princeton University Press, Princeton, N.J., 1983.

\bibitem{WHW}
M.~Weber, D.~Hoffman, and M.~Wolf, \emph{An embedded genus-one helicoid}, Ann.
  of Math. (2) \textbf{169} (2009), no.~2, 347--448.

\end{thebibliography}

\end{document}